\documentclass[leqno,12pt]{amsart}
\usepackage{setspace}
\usepackage[english]{babel}
\usepackage{bussproofs}
\usepackage{csquotes}
\usepackage{dirtytalk}
\usepackage{mathtools}
\usepackage{mathrsfs}
\usepackage{latexsym}
\usepackage{enumitem}
\usepackage{amsthm}
\usepackage{amssymb}
\DeclareMathAlphabet{\mathpzc}{OT1}{pzc}{m}{it}
\usepackage[latin1]{inputenc}
\usepackage{afterpage}

\usepackage{bbm}
\usepackage[rgb,dvipsnames]{xcolor}
\usepackage{tikz} 
\usetikzlibrary{decorations.text}
\usetikzlibrary{arrows,arrows.meta,hobby}
\usetikzlibrary{shapes.misc, positioning}

\usepackage{tikz-cd}

\usetikzlibrary{cd}

\newtheorem{theorem}{Theorem}[section]
\newtheorem{maintheorem}{Main Theorem}[section]
\newtheorem{conjecture}[theorem]{Conjecture}
\newtheorem{lemma}[theorem]{Lemma}
\newtheorem{proposition}[theorem]{Proposition}
\newtheorem{corollary}[theorem]{Corollary}

\newtheorem{claim}[theorem]{Claim}
\theoremstyle{definition}
\newtheorem{definition}[theorem]{Definition}

\theoremstyle{remark}

\newtheorem{question}[theorem]{Question}

\def\forces{\Vdash}

\def\ZFC{\mathsf{ZFC}}

\def\CH {\mathsf{CH}}

\def\P{\mathbb P}

\usepackage[margin=1.3in]{geometry}

\begin{document}

\title{Reflection Properties of Ordinals in Generic Extensions}
\author[Aguilera]{Juan P. Aguilera}
\address[J.~P.~Aguilera]{Kurt G\"odel Research Center, Institute of Mathematics, University of Vienna. Kolingasse 14, 1090 Vienna, Austria. Institute of Discrete Mathematics and Geometry, Vienna University of Technology.
Wiedner Hauptstrasse 8-10, 1040 Vienna, Austria.}
\email{aguilera@logic.at}
\author[Switzer]{Corey Bacal Switzer}
\address[C.~B.~Switzer]{Kurt G\"odel Research Center, Institute of Mathematics, University of Vienna. Kolingasse 14, 1090 Vienna, Austria.}
\email{corey.bacal.switzer@univie.ac.at}
\maketitle

\begin{abstract}
 We study the question of when a given countable ordinal $\alpha$ is $\Sigma^1_n$- or $\Pi^1_n$-reflecting in models which are neither $\mathsf{PD}$ models nor the constructible universe, focusing on generic extensions of $L$. We prove, amongst other things, that adding any number of Cohen or random reals, or forcing with Sacks forcing or any lightface Borel weakly homogeneous ccc forcing notion cannot change such reflection properties. Moreover we show that collapse forcing increases the value of the least reflecting ordinals but, curiously, to ordinals which are still smaller than the $\omega_1$ of $L$.
\end{abstract}

\section{Introduction}\label{SectIntro}
Let $n < \omega$. An ordinal $\alpha$ is called $\Sigma^1_n$-{\em reflecting} (respectively $\Pi^1_n$-{\em reflecting}) if for each tuple $\beta_1 < \beta_2 < \dots < \beta_l < \alpha$ and each $\Sigma^1_n$ (respectively $\Pi^1_n$) formula $\varphi(x)$ if $L_\alpha \models \varphi (\vec\beta)$ then there is an $\alpha' \in (\beta_l, \alpha)$ so that $L_{\alpha'} \models \varphi (\vec\beta)$. Here ``$L_\alpha \models \varphi (\vec\beta)$" and similar such sentences refer to full second order logic over $L_\alpha$ i.e.~quantifiers range over the full powerset of the structure, and in particular whether $L_\alpha \models \varphi (\vec\beta)$ is not absolute across different models of set theory. 
The least $\Sigma^1_n$-reflecting ordinal is denoted $\sigma^1_n$ and the least $\Pi^1_n$-reflecting ordinal is denoted $\pi^1_n$. These and all the ordinals of interest here will be such that the corresponding levels of $L$ are closed under pairing, and so for our purposes -- and in order to simplify notation -- we may always act as though $l = 1$ in the definition above.

Reflecting ordinals of this and other kinds have been considered since the 70s, most notably in the work of Aczel and Richter \cite{AcRi74}, where an overview of the history until then can be found. Building on a characterization of theirs, Aanderaa \cite{Aa74} proved that $\pi^1_1 < \sigma^1_1$. Richter \cite{Ri75} then showed that $\sigma^1_2 < \pi^1_2$.
It is not too hard to show, using Shoenfield's absoluteness theorem, that the ordinals $\pi^1_1$, $\sigma^1_1$, $\pi^1_2$ and $\sigma^1_2$ are correctly computed by $L$.

Once one goes beyond $\Sigma^1_2$ and $\Pi^1_2$, the situation becomes less clear and indeed depends on underlying set-theoretic assumptions. In particular, it is independent of $\mathsf{ZFC}$ whether $\sigma^1_{n+3}$ is smaller than or greater than $\pi^1_{n+3}$: on the one hand, Cutland \cite{Cu80} proved that $\sigma^1_{n+3} < \pi^1_{n+3}$ holds for all $n\in\mathbb{N}$ under the assumption that $V = L$; on the other hand, Kechris \cite{Ke75} established, via descriptive-set-theoretic methods, that Projective Determinacy implies that $\sigma^1_{2n} < \pi^1_{2n}$ and $\pi^1_{2n+1} < \sigma^1_{2n+1}$ hold for all $n\in\mathbb{N}$. 
An alternative proof of Kechris' theorem using Inner Model Theory is presented in \cite{AgRLPO}.

The main purpose of this article is to investigate how these ordinals compare to one another in models other than those satisfying $V=L$ or $\mathsf{PD}$. 
In particular, we develop methods to compute $\sigma^1_n$ and $\pi^1_n$ in generic extensions of $L$. We show that many standard ``definable" forcing notions do not change reflection properties of countable ordinals. As a result we get that forcing with such forcing notions over $L$ preserve the property of being $\sigma^1_n$ or $\pi^1_n$ for a countable ordinal. A sample of these results is as follows:

\begin{maintheorem}\label{TheoremMain1}
    For all $n < \omega$ in any generic extension of $L$ by $\P$ we have that $\sigma^1_n = (\sigma^1_n)^L$ and $\pi^1_n = (\pi^1_n)^L$ where $\P$ is any of the following forcing notions:
    \begin{enumerate}
        \item adding any number of Cohen reals
        \item adding any number of random reals
        \item any lightface ccc weakly homogeneous Borel forcing notion
        \item Sacks forcing
        \item Miller forcing
        \item Laver forcing
    \end{enumerate}
\end{maintheorem}
The precise definitions of these forcing notions are given in the sections where the instance of the theorem above is proved.

Theorem \ref{TheoremMain1} establishes that the values of $\sigma^1_n$ and $\pi^1_n$ are not changed by many of the familiar forcing notions commonly studied in the set theory of the real numbers. In particular, it shows that the statement $\sigma^1_3 < \pi^1_3$ is consistent with the negation of the Continuum Hypothesis.
One might thus wonder whether \textit{any} forcing notion over $L$ is able to modify the value of the ordinals $\sigma^1_n$ and $\pi^1_n$. Our other main result gives a very counterintuitive answer to this question: when forcing with L\'evy's partial order to collapse cardinals, the values of $\sigma^1_n$ and $\pi^1_n$ do increase, for $n \geq 3$. \textit{However}, they remain smaller than the size of $\omega_1^L$.

\begin{maintheorem}
Assume $V = L$. Let $\P$ be the standard forcing to collapse $\omega_1$ to be countable. If $G \subseteq \P$ is generic over $L$ then in $L[G]$ for each $2 < n,m < \omega$ we have that 
    \[(\sigma^1_n)^L < (\pi^1_n)^L < \sigma^1_m < \pi^1_m < (\omega_1)^L < \omega_1.\]
\end{maintheorem}

\section{Some Basic Facts and Preliminaries}
In this section we record some basic facts about reflecting ordinals that are used throughout. For the most part these results are folklore though we try to give references when possible. First we note that the least $\Sigma^1_n$ (respectively $\Pi^1_n$)-reflecting ordinal (for any $n < \omega$) is countable. 

\begin{proposition}
    For all $n < \omega$ the ordinals $\sigma^1_n$ and $\pi^1_n$ are countable. Indeed the set of $\Sigma^1_n$-reflecting ordinals and the set of $\Pi^1_n$-reflecting ordinals both contain a club in $\omega_1$.
\end{proposition}

\begin{proof}
    Fix $n < \omega$ and $\Gamma \in \{\Sigma, \Pi\}$ and suppose towards a contradiction that the set of $\Gamma^1_n$-reflecting ordinals does not contain a club in $\omega_1$. Then there is a stationary set $S \subseteq \omega_1$ of points which are not $\Gamma^1_n$-reflecting. For each $\alpha \in S$ let $\beta (\alpha) < \alpha$ be least so that there is a $\Gamma^1_n$-formula $\varphi(x)$ so that $L_\alpha \models \varphi (\beta(\alpha))$ but no smaller $L_\gamma$ is a model of $\varphi(\beta(\alpha))$. By Fodor's pressing down lemma there is a stationary $S' \subseteq  S$ so that $\beta = \beta(\alpha)$ for some fixed $\beta$ and every $\alpha \in S'$. Moreover since there are only countably many $\Gamma^1_n$-formulas there is again an $S'' \subseteq S'$ stationary and a fixed $\Gamma^1_n$-formula $\varphi(x)$ so that for each $\alpha \in S''$ we have that $L_\alpha \models \varphi(\beta)$ but no small $L_\gamma$ is a model of $\varphi (\beta)$. Of course now we have reached a contradiction by the defining property of $S''$.
\end{proof}

If $M$ is some transitive (possibly class sized) model of set theory then denote by $(\sigma^1_n)^M$ (respectively $(\pi^1_n)^M$) the ordinal as computed in that model. 

\begin{lemma}
    Fix $n < \omega$ at least 3 and $\Gamma \in \{\Sigma, \Pi\}$. For any $\alpha < \omega_1$ if $L \models$ ``$\alpha$ is not $\Gamma^1_n$-reflecting" then $\alpha$ is not $\Gamma^1_n$-reflecting. In particular for each $n < \omega$ we have that $\sigma^1_n \geq (\sigma^1_n)^L$ and $\pi^1_n \geq (\pi^1_n)^L$.\label{dontgodown}
\end{lemma}

\begin{proof}
    Fix $n < \omega$ at least 3 and $\Gamma \in \{\Sigma, \Pi\}$ and suppose $\alpha$ is an ordinal which is not $\Gamma^1_n$-reflecting in $L$. We will show that $\alpha$ is not $\Gamma^1_n$-reflecting in $V$. Oddly, there are two cases, not depending on whether $\Gamma$ is $\Pi$ or $\Sigma$ but rather on what the innermost quantifier is i.e. whether $\Gamma =\Pi$ and $n$ is odd or $\Gamma = \Sigma$ and $n$ is even - so the innermost quantifier is $\forall$ (case 1) or $\Gamma = \Pi$ and $n$ is even or $\Gamma = \Sigma$ and $n$ is odd so the innermost quantifier is $\exists$ (case 2). We take these one at a time. First though, since $\alpha$ is not $\Gamma^1_n$-reflecting in $L$ there is a $\Gamma^1_n$-formula $\varphi(x)$ and an ordinal $\beta < \alpha$ so that $L \models$``$L_\alpha \models \varphi (\beta)$" but for each $\gamma \in (\beta, \alpha)$ we have $L\models$ ``$L_\gamma$ does not model $\varphi(\beta)$". Fix such a $\varphi$ and $\beta$.\smallskip

    \noindent \underline{Case 1}: The innermost quantifier is $\forall$. For readability let us assume that $\Gamma = \Sigma$. Since we only care about the innermost quantifier nothing is changed in the argument by adding a vapid existential quantifier on the outside hence there is no loss of generality in assuming this. Let us be clear about the form of $\varphi(x)$ by writing it as follows: $$\varphi (x) = \exists X_0\hdots\exists X_{n-2}\forall X_{n-1} \psi (X_0, \hdots, X_{n-1}, x)$$
where $\psi$ is quantifier free. Let $\varphi^*(x)$ then be the following partial relativization of $\varphi(x)$ to $L$: $$\exists X_0\Big[X_0 \in L \land \forall X_1 \big[X_1 \notin L \lor \hdots\exists X_{n-2}[X_{n-2} \in L \land \forall X_{n-1} \psi (X_0, \hdots, X_{n-1}, x)]\big]\hdots\Big]$$
To be clear, each quantifier except the final one is relativized to $L$. Note that if $V=L$ then $\varphi^*(\beta)$ is equivalent to $\varphi(\beta)$. Since the statement ``$X_i \in L$" is $\Sigma^1_2$ counting quantifiers reveals that $\varphi^*(x)$ is also $\Gamma^1_n$ (this part uses that $n \geq 3$). This is the formula that we will show is not reflected by $\alpha$.  There are two parts to this. We need to show first that $L_\alpha \models \varphi^*(\beta)$ (in $V$) and second that no smaller $L_\gamma$ is a model of $\varphi^*(\beta)$. We take these one at a time. 

First observe that since $L \models$ ``$L_\alpha \models \varphi (\beta)$" we have that $L \models$ ``$L_\alpha \models \varphi^*(\beta)$". Thus in $V$ we have that $\exists X_0 \in L$ such that for all $X_1 \in L\hdots\exists X_{n-2} \in L$ so that $L \models$ ``$L_\alpha \models \forall X_{n-1} \psi (X_0, \hdots, X_{n-1}, \beta)$" but now ``$L_\alpha \models \forall X_{n-1} \psi (X_0, \hdots, X_{n-1}, \beta)$" is a $\Pi^1_1$ statement with parameters in $L$ so it is upwards absolute by Shoenfield absoluteness. Thus in $V$ we have that $\exists X_0 \in L$ such that for all $X_1 \in L\hdots\exists X_{n-2} \in L$ so that $L_\alpha \models \forall X_{n-1} \psi (X_0, \hdots, X_{n-1}, \beta)$ which means that $L_\alpha \models \varphi^*(\beta)$. 

Now we show that no smaller $L_\gamma$ is a model of $\varphi^*(\beta)$. Indeed suppose some $L_\gamma \models \varphi^* (\beta)$. Then there is an $X_0 \in L$ so that for each $X_1 \in L \hdots$ so that there is an $X_{n-2} \in L$ so that $L_\gamma \models \forall X_{n-1} \psi (X_0, \hdots, X_{n-1}, \beta)$. But now $L_\gamma \models \forall X_{n-1} \psi (X_0, \hdots, X_{n-1}, \beta)$ is a $\Pi^1_1$ statement with parameters in $L$ so it is downwards absolute, which means that $L \models$``$L_\gamma \models \varphi^*(\beta)$" and hence $L\models$ ``$L_\gamma \models \varphi (\beta)$" contradicting the assumption on the defining property of $\varphi$ and $\beta$. \smallskip

\noindent \underline{Case 2}: The innermost quantifier is $\exists$. This case is nearly the same so we focus on the differences. Again we assume the outermost quantifier of $\varphi (x)$ is $\exists$ and it is of the form $$\varphi (x) = \exists X_0\hdots\forall X_{n-2}\exists X_{n-1} \psi (X_0, \hdots, X_{n-1}, x)$$
with $\psi$ quantifier free. Let $\varphi^*(x)$ be the same almost relativization: $$\exists X_0\Big[X_0 \in L \land \forall X_1 \big[X_1 \notin L \lor \hdots\forall X_{n-2}[X_{n-2} \notin L \lor \exists X_{n-1} \psi (X_0, \hdots, X_{n-1}, x)]\big]\hdots\Big].$$

The proof now proceeds essentially as before to show that $L_\alpha \models \varphi^*(\beta)$ but no smaller $L_\gamma$ is a model of $\varphi^*(\beta)$. The difference is now that we use the simple fact that $\Sigma^1_1$-statements are upwards absolute to ensure that $L_\alpha \models \varphi^*(\beta)$ in $V$ and Shoenfield absoluteness to show that no smaller $L_\gamma$ models $\varphi^*(\beta)$. 
    \end{proof}

\section{Cohen Forcing}

In this section we study the preservation of reflection properties of ordinals after adding arbitrarily many Cohen reals. For a set $X$ denote by $\mathbb C_X$ the forcing notion for adding $|X|$-Cohen reals indexed by $X$ i.e.~conditions in $\mathbb C_X$ are finite, partial functions $p: X \times \omega \to 2$ reverse ordered by inclusion. It's well known that for $X$ and $Y$ of the same cardinality $\mathbb C_X$ and $\mathbb C_Y$ are forcing equivalent. Recall also that any countable, non-trivial forcing notion is forcing equivalent to adding one Cohen real. The main theorem we will prove in this section is the following.

\begin{theorem}
    Let $\kappa$ be an arbitrary cardinal, $n < \omega$ and $\Gamma \in \{\Sigma, \Pi\}$. If $\alpha < \omega_1$ is $\Gamma^1_n$-reflecting then $\forces_{\mathbb C_\kappa}$``$\alpha$ is $\Gamma^1_n$-reflecting". \label{Cohen}
\end{theorem}

As explained in paragraph preceeding the statement of Theorem \ref{Cohen} we could have written $\mathbb C_X$ for an arbitrary set $X$ in lieu of $\mathbb C_\kappa$ for a cardinal $\kappa$. The immediate corollary of this theorem is the following, which is what we were initially interested in.

\begin{corollary}
    In any generic extension of $L$ by any number of Cohen reals $\sigma^1_n = (\sigma^1_n)^L$ and $\pi^1_n = (\pi^1_n)^L$ for all $n < \omega$. In particular the inequality $\sigma^1_{n+2} < \pi^1_{n+2}$ holds.
\end{corollary}

Let us first prove the Corollary assuming Theorem \ref{Cohen}.

\begin{proof}
    Fix $n < \omega$, $\kappa$ a cardinal and work in $L[G]$ for some generic $G \subseteq \mathbb C_\kappa$. By Lemma \ref{dontgodown} we have that $\sigma^1_n \geq (\sigma^1_n)^L$ and same for $\pi^1_n$. By Theorem \ref{Cohen} we get that $(\sigma^1_n)^L$ is still $\Sigma^1_n$-reflecting (and same for $\pi^1_n$/$\Pi^1_n$) which completes the proof. 
\end{proof}

Let us now turn to the proof of Theorem \ref{Cohen}. As a warm-up which introduces most of the important ideas in the proof let us first prove that Theorem \ref{Cohen} holds in the special case that $\kappa = 1$. For brevity let us refer to $\mathbb C_1$ as $\mathbb C$, which is the standard forcing notion to add one Cohen real. We take this for definiteness as finite partial functions from $\omega$ to $2$, which we treat as natural numbers. Note that as such a subset of $\mathbb C$ is itself a real. To say then that $A \subseteq \mathbb C$ is a maximal antichain is arithmetical in $A$ and similarly for other standard forcing properties such as dense, predense etc.

\begin{lemma}
    Let $n < \omega$ and $\Gamma \in \{\Sigma, \Pi\}$. If $\alpha < \omega_1$ is $\Gamma^1_n$-reflecting then we have $\forces_{\mathbb C}$``$\alpha$ is $\Gamma^1_n$-reflecting". 
\end{lemma}

\begin{proof}
    Fix $n < \omega$ and $\Gamma \in \{\Sigma, \Pi\}$. Let $\alpha < \omega_1$ be $\Gamma^1_n$-reflecting. Suppose towards a contradiction that it is forced that, for some fixed $\beta < \alpha$ and $\Gamma^1_n$-formula $\varphi$ that $L_\alpha \models \varphi(\beta)$ but no smaller $L_\gamma$ models $\varphi(\beta)$. The argument actually depends on which $n$ and $\Gamma$ almost not at all so let us consider the case $\Gamma = \Sigma^1_3$ and leave the cosmetic modifications for the other cases to the reader. Let us be clear about the form of $\varphi(x)$ and write it as $\exists X \forall Y \exists Z \psi (X, Y, Z, x)$ with $\psi$ second order quantifier free. 

    Before continuing, let us make a remark on $\mathbb C$-names for reals, or indeed subsets of $L_\alpha$ from the point of view of the second order logic over $L_\alpha$. By the ccc, plus the fact that, since $\mathbb C$ is itself an arithmetical set of natural numbers and hence $\mathbb C \in L_\alpha$ and correctly computed for any $\alpha$, each $\dot{X}$ for a subset of $L_\alpha$ is forced to be equal to a name of the following form. For each $x \in L_\alpha$ let $A_x \subseteq \mathbb C$ be a maximal antichain of elements deciding $\check{x} \in \dot{X}$. Then $\dot{X}$ is equivalent to the name $$\dot{X}':=\bigcup_{x \in L_\alpha} \bigcup_{p \in A_x} \{\langle p, x, i\rangle \} $$ where $i \in 2$ and $\langle p, x, 1\rangle \in \dot{X} '$ if and only if $p \forces \check{x} \in \dot{X}$. Note that this itself is a subset of $L_\alpha$. Moreover to say that some subset $W \subseteq L_\alpha$ is a Cohen name like this is first order since it only quantifies over elements of $L_\alpha$ (note that for Cohen forcing, conditions are elements of $L_\alpha$). Let us denote by $\mathbb C$-$\mathsf{NAME}$ the arithmetical set of such names. 
        
        Now consider the statement (as evaluated as a second order statement in $L_\alpha$)  ``$\forall p \in \mathbb C$ $\exists \dot{X} \in \mathbb C$-$\mathsf{NAME}$ $\forall \dot{Y} \in \mathbb C$-$\mathsf{NAME}$ $\exists \dot{Z} \in \mathbb C$-$\mathsf{NAME}$ $p \forces \psi(\dot{X}, \dot{Y}, \dot{Z}, \check{\beta})$". Since the set $\mathbb C$-$\mathsf{NAME}$ is first order definable it suffices to show that ``$p \forces \psi(\dot{X}, \dot{Y}, \dot{Z}, \check{\beta})$" is a second order  quantifier free statement as in this case, by the $\Sigma^1_3$-reflecting-ness of $\alpha$ this would have to be true at some smaller $\beta$ contradicting the assumption on $\varphi(x)$. However this is simply a easy observation on the way forcing is defined using the fact that statements of the form ``$x \in \dot{X}$" etc are arithmetical since they are coded into the presentation of the name and otherwise the standard way of defining forcing inductively is second order quantifier free in this context. 
\end{proof}

The challenge of supping this proof up to prove Theorem \ref{Cohen} is that for an uncountable index set we can no longer describe the forcing conditions, relations, names etc as subsets of some countable $L_\alpha$ -- indeed the entire set up is no longer a priori expressible. To rectify this we will first reduce to the case that $\kappa = \omega_1$ and then show how to handle this case. For this first task note that if $\beta < \alpha$ are countable ordinals and $\varphi (x)$ is a $\Gamma^1_n$-formula for some $n < \omega$ and some $\Gamma \in \{\Pi, \Sigma\}$ then saying ``$L_\alpha \models \varphi(\beta)$" is $\Gamma^1_n$ in the sense of the projective hierarchy with parameter (the real coding) $L_\alpha$.

\begin{lemma}
    Let $X$ be an arbitrary uncountable index set. For any real $r \in 2^\omega$ and any projective statement $\varphi(x)$ we have that if $\varphi(\check{r})$ is forced by $\mathbb C_X$ if and only if it is forced by $\mathbb C_{\omega_1}$. \label{projectivetruthisstable}
\end{lemma}

The above says roughly that any two extensions by uncountably many Cohen reals agree about projective truth. Before proving this we recall briefly some important terminology. If $X$ is a set and $p \in \mathbb C_X$ is a condition then denote by ${\rm sup}(p)$ the {\em support of} $p$, i.e. the set of $i \in X$ so that for some $k < \omega$ the pair $\langle i, k\rangle$ is in the domain of $p$. Note this is finite. We can extend the idea of a support to a name by taking the union of the supports of the conditions that appear hereditarily in the name. If $\dot{x}$ is a $\mathbb C_X$-name then then the evaluation of $\dot{x}$ already appears in any generic extension by $\mathbb C_{{\rm sup}(\dot{x})}$. Note that if $\dot{x}$ is a nice name for a real then its support is countable. 

We will also need in the proof that Cohen forcing is {\em weakly homogeneous}, see Jech \cite[Theorem 26.12]{Je03} for a definition but what we need is that if $a \in V$ is in the ground model and $\Theta(a)$ is a formula in the language of set theory then the maximal condition of $\mathbb C_X$ forces $\Theta(\check{a})$ or $\neg \Theta (\check{a})$ (for any $X$). 

\begin{proof}

The proof is by induction on $n < \omega$, where $\varphi(r)$ is $\Sigma^1_n$ or $\Pi^1_n$. The base case and in fact up to $n=2$ is covered by Shoenfield absoluteness. Fix an integer $3 \leq n< \omega$ and suppose now that the lemma holds for every $\Sigma^1_n$ or $\Pi^1_n$ formula (and every real parameter $r$). There are two cases. First let $\varphi (x)$ be $\Sigma^1_{n+1}$ and write it as $\exists y \psi(x, y)$ with $\psi$ a $\Pi^1_n$ formula. First that some condition $p \in \mathbb C_X$ forces $\exists y \psi(\check{r}, y)$. By the ccc of Cohen forcing, there is a countable $Y \subseteq X$ so that ${\rm sup}(p) \subseteq Y$ and there is a $\mathbb C_Y$-name $\dot{y}$ so that $p\forces_{\mathbb C_X} \psi(\check{r}, \dot{y})$ (note that even though $\dot{y}$ is in the extension by $\mathbb C_Y$ the statement might not hold there). We can now decompose $\mathbb C_X$ as $\mathbb C_Y \times \mathbb C_{X\setminus Y}$. Let $G_Y \times G_{X\setminus Y}$ be generic for this poset. Let $y = \dot{y}^{G_Y} \in V[G_Y]$. Since $X$ is uncountable, $X \setminus Y$ is also uncountable. Now apply the inductive hypothesis in $V[G_Y]$ to $\psi(y, r)$ and $\mathbb C_{X \setminus Y}$. Namely, since $\psi$ is $\Pi^1_n$ we have that $V[G_Y] \models$``$\forces_{\mathbb C_{\omega_1}} \psi( \check{y}, \check{r})$". Thus $\psi(y, r)$ holds in $V[G_Y][G_{\omega_1}]$ for any generic $G_{\omega_1} \subseteq \mathbb C_{\omega_1}$. However forcing with $\mathbb C_Y \times \mathbb C_{\omega_1}$ is the same as forcing with $\mathbb C_{Y \cup \omega_1}$ (assuming without loss of generality that $Y \cap \omega_1 = \emptyset$) and the latter, being a set of size $\omega_1$ is the same as forcing over $V$ with $\mathbb C_{\omega_1}$, thus completing this case.

Next assume that no $p \in \mathbb C_X$ forces that $\exists y \psi(\check{r}, y)$. Then the maximal condition forces that $\forall y \neg \psi (\check{r}, y)$. We want to show that $\mathbb C_{\omega_1}$ also forces $\forall y \neg \psi (\check{r}, y)$. If this is not true then there is a $\mathbb C_{\omega_1}$-name $\dot{y}$ (which we can equally think of as a $\mathbb C_X$-name) with (countable) support $Y$ so that $\forces_{\mathbb C_{\omega_1}} \psi(\check{r}, \dot{y})$. Let $G_Y \subseteq\mathbb C_Y$ be generic and work in $V[G_Y]$ with $\dot{y}^{G_Y} = y$. Since $\omega_1 \setminus Y$ is still of size $\omega_1$ we have that in this model $\mathbb C_{\omega_1}$ forces that $\psi(\check{r}, \check{y})$ and, by the inductive hypothesis so does $\mathbb C_X$. But this means back in $V$ that $\mathbb C_X$ forces $\exists y \psi(\check{r}, y)$, which is a contradiction.

    We now turn to the case where $\varphi(x)$ is a $\Pi^1_{n+1}$ statement. Let $\varphi(x)$ be written as $\forall y \psi (x, y)$ where $\psi$ is a $\Sigma^1_{n}$-formula. If no condition in $\mathbb C_X$ forces $\forall y \psi (\check{r}, y)$ then the maximal condition forces $\exists y \neg \psi (\check{r}, y)$ and we reduce to the previous case. If there is a $p \in \mathbb C_X$ forcing $\forall y \psi (\check{r}, y)$ then either, by weak homogeneity, the maximal condition of $\mathbb C_{\omega_1}$ also forces $\forall y \psi (\check{r}, y)$, in which case we are done or else the maximal condition of $\mathbb C_{\omega_1}$ forces that $\exists y \neg \psi (\check{r}, y)$ in which case the proof from the previous case implies $\mathbb C_X$ forces $\exists y \neg \psi (\check{r}, y)$ as well, which is a contradiction.

\end{proof}

We now continue with the proof of Theorem \ref{Cohen}. In light of Lemma \ref{projectivetruthisstable} it suffices to consider the case that $X = \omega_1$, which we will do now without further comment. We need one more lemma about $\mathbb C_{\omega_1}$. Below given linear orders $A$ and $B$ write $A \leq B$ if $A$ can be order embedded into $B$. Generalizing the fact that if $\alpha < \beta$ are ordinals then $\mathbb C_\alpha$ is a complete suborder of $\mathbb C_\beta$, note that if $A \leq B$ then we can treat $\mathbb C_A$ as a regular suborder of $\mathbb C_B$ by extending any order embedding of $A$ into $B$ to one from $\mathbb C_A$ into $\mathbb C_B$. Without further comment we will therefore treat $\mathbb C_A$-names as $\mathbb C_B$-names in such situations.

\begin{lemma}
Let $r \in 2^\omega$ be a real and $n < \omega$. Let $\varphi (x)$ be a projective statement.
\begin{enumerate}
    \item If $\varphi (x)$ is a $\Sigma^1_n$-formula of the form $\exists y_0 \forall y_1 \hdots \psi(x, y_0, y_1, \hdots)$ then $\forces_{\mathbb C_{\omega_1}} \varphi(\check{r})$ if and only if there is a countable linear order $A_0$ and a $\mathbb C_{A_0}$-name $\dot{y}_0$ so that for each countable linear order $A_1$ and each $\mathbb C_{A_1}$-name $\dot{y}_1$ if $A_0 \leq A_1$ then there is a countable linear order $A_2$ so that $A_1 \leq A_2$ and a $\mathbb C_{A_2}$-name $\dot{y}_2$ so that $\hdots$ so that $\forces_{\mathbb C_{A_{n-1}}} \psi(\check{r}, \dot{y}_0, \dot{y}_1, \hdots)$.
    \item If $\varphi (x)$ is a $\Pi^1_n$-formula of the form $\forall y_0 \exists y_1 \hdots \varphi(r, y_0, y_1, \hdots)$ then $\forces_{\mathbb C_{\omega_1}} \varphi(\check{r})$ if and only if for every countable linear order $A_0$ and every $\mathbb C_{A_0}$-name $\dot{y}_0$ there is a countable linear order $A_1$ and a $\mathbb C_{A_1}$-name $\dot{y}_1$ so that $A_0 \leq A_1$ and for every countable linear order $A_2$ and every $\mathbb C_{A_2}$-name $\dot{y}_2$ if $A_1 \leq A_2$ then $\hdots$ so that $\forces_{\mathbb C_{A_{n-1}}} \psi(\check{r}, \dot{y}_0, \dot{y}_1, \hdots)$.
    \end{enumerate}
    \label{LOlemma}
\end{lemma}

Note that despite the verbose phrasing this lemma is essentially a rephrasing of the well known fact that every boldface $\Delta^1_1$ statement forced by adding $\omega_1$-many Cohen reals is already true after adding some initial $\alpha$-many. The use of linear orders is somewhat unnecessary and could be replaced by e.g. arbitrary countable sets and their supersets. However the former is useful for coding which we will use to finish the proof of Theorem \ref{Cohen}. In any case the point of $A_i \leq A_{i+1}$ is just so that we can treat $\dot{y}_i$ as a $\mathbb C_{A_{i+1}}$-name.

\begin{proof}
    We will prove items (1) and (2) at the same time as their proofs are intertwined. This is because, by the weak homogeneity of Cohen forcing, the backward direction of (2) is essentially the same as the forward direction of (1) and vice versa. The proof is by induction on $n < \omega$. The cases $n \leq 2$ are by Shoenfield absoluteness noting that if $A$ is a countable linear order than $\mathbb C_A$ can be treated as a regular suborder of $\mathbb C_{\omega_1}$. Fix $3 \leq n < \omega$, a real $r \in 2^\omega$ and a $\Sigma^1_n$ formula $\varphi(x)$ which we write as $\exists y_0 \forall y_1 \hdots \psi(x, y_0, y_1, \hdots)$. Suppose first that $\forces_{\mathbb C_{\omega_1}} \varphi(\check{r})$ holds. This means that there is a $\mathbb C_{\omega_1}$-name $\dot{y}_0$ so that $\forces_{\mathbb C_{\omega_1}}\forall y_1 \hdots \psi(\check{r}, \dot{r}_0, y_1, \hdots)$. By the ccc there is a countable $A \subseteq \omega_1$ so that $\dot{y}_0$ is equivalent to a $\mathbb C_A$-name. Taking $A_0 = A$ (with the induced suborder from $\omega_1$) and applying the inductive hypothesis to (2) in $V^{\mathbb C_A}$ completes the proof of this case. 
    
    For the other direction assume that $\mathbb C_{\omega_1}$ does not force $\varphi(\check{r})$ and note that this implies that that $\forces_{\mathbb C_{\omega_1}} \neg \varphi (\check{r})$ i.e. $\forces_{\mathbb C_{\omega_1}} \forall y_0, \hdots\neg\psi (\check{r}, y_0, \hdots))$. Thus for each $\mathbb C_{\omega_1}$-name $\dot{y}_0$ we have that $\forces_{\mathbb C_{\omega_1}} \exists y_1, \hdots\neg\psi (\check{r}, \dot{y}_0, y_1, \hdots))$. However any such name is equivalent to a $\mathbb C_{Y}$-name for some countable $Y \subseteq X$. Since the order on $Y$ does not matter we can order it however we want. Alongside the inductive assumption this completes the proof.
\end{proof}

We now complete the proof of Theorem \ref{Cohen}. 

\begin{proof}[Proof of Theorem \ref{Cohen}]
We can code countable linear orders as reals by treating their domains as $\omega$ in the standard way. Let $LO$ be the set of all such reals. Note that this set is arithmetic. Moreover if $A, B \in LO$ then the statement $A \leq B$ is well known and easily verified to be $\Sigma^1_1$. We will use these facts in the proof without further comment.

Fix $n < \omega$ and $\Gamma \in \{\Sigma, \Pi\}$. Let $\alpha < \omega_1$ be $\Gamma^1_n$-reflecting. By Lemma \ref{projectivetruthisstable} it suffices to show that $\mathbb C_{\omega_1}$ forces that $\alpha$ is $\Gamma^1_n$-reflecting. Suppose towards a contradiction that it is forced by $\mathbb C_{\omega_1}$ that for some fixed $\beta < \alpha$ and $\Gamma^1_n$-formula $\varphi$ that $L_\alpha \models \varphi(\beta)$ but no smaller $L_\gamma$ models $\varphi(\beta)$. As in the case of one Cohen real, the argument actually depends on which $n$ and $\Gamma$ almost not at all so let us consider the case $\Gamma = \Sigma^1_3$ and leave the cosmetic modifications for the other cases to the reader. Let us be clear about the form of $\varphi(x)$ and write it as $\exists X \forall Y \exists Z \psi (X, Y, Z, x)$ with $\psi$ second order quantifier free. By Lemma \ref{LOlemma} this is equivalent to the following statement:
\begin{align*}
\exists A \in LO\, &\exists \dot{X} \in \mathbb C_{A}{\rm -}\mathsf{NAME}\,  \Big[ \forall B \in LO\, \forall \dot{Y} \in \mathbb C_{B}{\rm -}\mathsf{NAME}\, \big[ A \leq B \\
&\to \exists C \in LO\, \exists \dot{Z} \in \mathbb C_{C}\, [ B \leq C \land \forces_{\mathbb C_C} \check{L}_\alpha \models \psi(\dot{X}, \dot{Y}, \dot{Z}, \check{\beta})]\big]\Big]
\end{align*}

    Counting quantifiers and applying coding in the way we have done this whole section however reveals this to be a $\Sigma^1_3$-statement about $L_\alpha$ which must therefore reflect to some smaller $L_\gamma$, contradicting the assumption.
\end{proof}

\section{Borel ccc Forcing}
Theorem \ref{Cohen} has a larger generalization in the case that $\kappa = 1$ (or $X$ is a Borel subset of a Polish space actually). Recall that a forcing notion $\P$ is called {\em Borel} if the underlying set of $\P$ as well as the order and incompatibility relations are Borel sets in some appropriate Polish space. Note that in this case the compatibility relation is also Borel. For the purposes of this section let us call a forcing notion {\em Borel+} if it is Borel and has the additional property that for each projective formula $\varphi(x)$ and every ground-model real $a$ either the maximal condition of $\P$ forces $\varphi(\check{a})$ or its negation. Note that many natural Borel ccc forcing notions have this property such as random forcing and Hechler forcing. The main theorem of this section is the following.

\begin{theorem}
     Let $\P$ be a lightface Borel+ ccc forcing notion. Let $n < \omega$ and $\Gamma \in \{\Pi, \Sigma\}$. For any set $X$ and any countable ordinal $\alpha$ if $\alpha$ is $\Gamma^1_n$-reflecting then it is forced to be so by $\P$. \label{Borel}
\end{theorem}

We get the following immediate corollary exactly as before.

\begin{corollary}
    For any lightface Borel+ ccc forcing notion $\P$, all countable $\alpha$, all $n < \omega$ and all $\Gamma \in \{\Pi, \Sigma\}$, $\alpha$ is $\Gamma^1_n$-reflecting in $L$ if and only if it is $\Gamma^1_n$-reflecting in $L[G]$ for any generic $G \subseteq \mathbb P$.
\end{corollary}

We turn to the proof of Theorem \ref{Borel}. From now on fix a lightface Borel+ ccc forcing notion $\P$. Suppose $\dot{X}$ names a subset of $\omega$ (or a real). Then, in exactly the way described for Cohen forcing above, there is an equivalent name, a {\em nice name} $\dot{X}'$, which can be written as follows. For each $n < \omega$ let $A_n$ be a maximal antichain of conditions deciding $\check{n} \in \dot{X}$. Note that this is countable and label its elements $\{q^n_k\; |\; k < \omega\}$. Now write 
\[\dot{X}'=\bigcup_{n \in \omega} \big\{ \langle q^n_k, n, i\rangle \; | \; q^n_k \forces \check{n} \text{ if and only if }i = 0 \big\}.\] From now on we always only consider such names. Note that prima facie to say that a given (countable) set is a maximal antichain is coanalytic hence the set of nice names for reals can be coded as a coanalytic set. 
Moreover observe then that to say that for some condition $p \in \P$ that $p \forces \check{n} \in \dot{X}$ is $\Delta^1_1$ with parameters $p$ and $\dot{X}$ since $p$ forces $\check{n} \in \dot{X}$ if and only if for each $k < \omega$ if $p$ is compatible with $q^n_k$ then $\langle q^n_k, n, 0\rangle \in \dot{X}$. By induction on length of the formula this can be bootstrapped to easily show that ``$p \forces \varphi (\dot{X})$" is $\Delta^1_1$ with parameters $p$ and $\dot{X}$ for any arithmetical formula $\varphi$. We will use this moving forward without further comment. We begin with some simple but key lemmas.

\begin{lemma}
    Suppose $\varphi(x)$ is a $\Pi^1_2$ formula, $\P$ is a lightface Borel+ ccc forcing notion, $\alpha$ is a countable ordinal and $\dot{X}$ is a $\P$-name for a real. The statement ``$\forces_\P L_\alpha \models \varphi(\dot{X})$" is $\Pi^1_2(\alpha, \dot{X})$. \label{pi12}
\end{lemma}

Here to be clear, we are treating $\dot{X}$ by the real coding it in the ground model and that is the parameter we are referring to, \textit{not} the evaluation in the generic extension. Before proving the lemma let us note what the issue is. A priori ``$\forces_\P L_\alpha \models \varphi(\dot{X})$" is equivalent to saying for all $\P$-names $\dot{Y}$ there exists a $\P$-name $\dot{Z}$ so that $\forces_\P L_\alpha \models \psi (\dot{X}, \dot{Y}, \dot {Z})$ where $\varphi(x)$ is of the form $\forall Y \exists Z \psi(x, Y, Z)$. Since being a $\P$-name is $\Pi^1_1$ this statement is prima facie $\Pi^1_3$ hence the content of the lemma consists of showing how to strip the innermost quantifier. 

\begin{proof}
    Fix $\varphi(x)$, $\P$, $\dot{X}$ and $\alpha$ as in the statement of the Lemma. Let $\psi$ be quantifer free so that $\varphi (x)$ has the form $\forall Y \exists Z \psi(x, Y, Z)$. Let $G \subseteq \P$ be generic and work briefly in $V[G]$. Let $\dot{X}^G = X$. Since $\forces_\P L_\alpha \models \varphi(\dot{X})$ for every $Y$ the statement $L_\alpha \models \exists Z \psi (X, Y, Z)$ holds in $V[G]$. Note that this is an analytic statement which therefore relativizes to any transitive set containing all of the requisite reals. In particular, if $\beta > \alpha$ and $L_\beta \models \mathsf{ZF}^-$ then $L_\beta [\dot{X}, \dot{Y}][G]$ is already a model of $L_\alpha \models \exists Z \psi (X, Y, Z)$ for any $\P$-name $\dot{Y}$ so that $\dot{Y}^G = Y$. Thus there is a real $Z \in L_\beta[\dot{X}, \dot{Y}][G]$ which (correctly) witnesses this fact.
    
    Fix $\beta$ and $\dot{Y}$ as above and move back to $V$. Since $\P$ is ccc if $A \in L_\beta[\dot{X}, \dot{Y}]$ is such that $L_\beta[\dot{X}, \dot{Y}] \models$``$A$ is a maximal antichain in $\P$" then this is correct as $\Pi^1_1$-statements are upwards absolute. Moreover $A \subseteq L_\beta[\dot{X}, \dot{Y}]$. It follows that $L_\beta[\dot{X}, \dot{Y}][G]$ is actually a generic extension of $L_\beta[\dot{X}, \dot{Y}]$ hence there is a name $\dot{Z} \in L_\beta[\dot{X}, \dot{Y}]$ for the real $Z$ we found above in $L_\beta[\dot{X}, \dot{Y}][G]$. Moreover since being a name is also $\Pi^1_1$ the model is also correct about this. 

    Putting all of this together we get that $\forces_\P L_\alpha \models \varphi(\dot{X})$ if and only if for all $\P$-names $\dot{Y}$ and all $\beta > \alpha$ and all well founded $M \cong L_\beta [\dot{X}, \dot{Y}] \models \mathsf{ZF}^-$ there is a $\dot{Z} \in M$ which $M$ thinks is a $\P$-name so that $\forces_\P \psi(\dot{X}, \dot{Y}, \dot{Z})$. We leave it to the reader to verify that this later statement is in fact $\Pi^1_2$ however note that the point is that quantifying over elements of $M$ or what $M$ models adds no quantifier complexity to the statement and this is how the innermost quantifier is stripped away. 
\end{proof}

We have a similar statement for $\Sigma^1_2$ formulas.
\begin{lemma}
    Suppose $\varphi(x)$ is a $\Sigma^1_2$ formula, $\P$ is a lightface Borel+ ccc forcing notion, $\alpha$ is a countable ordinal and $\dot{X}$ is a $\P$-name for a real. The statement ``$\forces_\P L_\alpha \models \varphi(\dot{X})$" is $\Sigma^1_2(\alpha, \dot{X})$. \label{sigma12}
\end{lemma}

\begin{proof}
    The proof is almost the same as that of Lemma \ref{pi12} so we sketch the differences and leave the details to the exacting reader. Fix $\varphi(x)$, $\P$, $\dot{X}$ and $\alpha$ as in the statement of the Lemma. Let $\psi$ be quantifer free so that $\varphi (x)$ has the form $\exists Y \forall Z \psi(x, Y, Z)$. Let $G \subseteq \P$ be generic and work briefly in $V[G]$. As before we obtain that if $\beta > \alpha$ and $L_\beta \models \mathsf{ZF}^-$ then $L_\beta [\dot{X}, \dot{Y}][G]$ is already a model of $L_\alpha \models \forall Z \psi (X, Y, Z)$ for any $\P$-name $\dot{Y}$ so that $\dot{Y}^G = Y$ is a witness to $L_\alpha \models \exists Y \forall Z \psi(X, Y, Z)$. Thus for every $Z \in L_\beta[\dot{X}, \dot{Y}][G]$ we have that $L_\beta[\dot{X}, \dot{Y}][G]$ correctly computes that $L_\alpha \models \psi(X, Y, Z)$ for all $Z$. Moreover if there is any $Z$ in any transitive model $M$ containing $X$ and $Y$ so that $M \models L_\alpha \neg \psi( X, Y, Z)$ there is such a $Z$ in every such $M$. 

    It follows that arguing as before we get that $\forces_\P L_\alpha \models \varphi(\dot{X})$ if and only if there is a $\P$-name $\dot{Y}$ and a $\beta > \alpha$ and a well founded $M \cong L_\beta [\dot{X}, \dot{Y}] \models \mathsf{ZF}^-$ so that for all $\dot{Z} \in M$ which $M$ thinks is a $\P$-name $M$ thinks that $\forces_\P \psi(\dot{X}, \dot{Y}, \dot{Z})$. Again since quantifying over $M$ adds no complexity this latter statement is indeed $\Sigma^1_2$.
    \end{proof}

We now can give the proof of Theorem \ref{Borel}.

\begin{proof}[Proof of Theorem \ref{Borel}]
    Fix a Borel+ forcing notion $\P$, a natural number $n < \omega$, a $\Gamma \in \{\Pi, \Sigma\}$ and a countable ordinal $\alpha$ which is $\Gamma^1_n$-reflecting. The reader will see in the proof that there is no loss of generality in assuming that $\Gamma = \Sigma$. The proof however does bifurcate into whether $n$ is odd or even and we take these one at a time. Suppose first that $n$ is odd. Note that this means that the final two quantifiers of any $\Sigma^1_n$-formula will be $\forall \exists$ in that order. Fix a $\beta < \alpha$ and a $\Sigma^1_n$-formula $\varphi (\beta)$ and suppose, $\P$ forces that $L_\alpha \models \varphi (\beta)$. Thus we can write $\varphi (\beta)$ as $\exists X_0 \forall X_1 \hdots \forall Y \exists Z \psi (X_0, X_1, \hdots, Y, Z, \beta)$.  But now $\forces_\P L_\alpha \models \varphi (\beta)$ is equivalent to saying there is a $\P$-name $\dot{X}_0$ so that for all $\P$-names $\dot{X}_1 \hdots\forces_\P L_\alpha \models \forall Y \exists Z \psi (\dot{X}_0, \dot{X}_1, \hdots, Y, Z, \beta)$. By Lemma \ref{pi12} this statement is itself $\Sigma^1_n$ as a second order statement about $L_\alpha$ hence by the reflecting property of $\alpha$ there is a $\gamma \in (\beta, \alpha)$ to which it reflects which completes this case.

    In the case that $n$ is even the proof is almost verbatim the same with the exception that the innermost two quantifiers are $\exists \forall$ and hence we apply Lemma \ref{sigma12}.
\end{proof}

The proof of the above can be augmented in some cases to yield forcing notions which produce models of the negation of $\CH$. In particular the case of adding arbitrarily many random reals. Recall that if $X$ is an infinite set, the forcing to add a set of random reals indexed by $X$ is just the set of positive Borel subsets of $2^{X \times \omega}$ with the product measure of the uniform measure on 2, see Jech \cite[Example 15.31]{Je03} for more details. We denote this forcing notion $\mathbb B_X$. Theorem \ref{Borel} can be improved in this case to show that Theorem \ref{Cohen} holds for random forcing as well.

\begin{theorem}
    Let $n < \omega$ and $\Gamma \in \{\Pi, \Sigma\}$. For any set $X$ and any countable ordinal $\alpha$ if $\alpha$ is $\Gamma^1_n$-reflecting then it is forced to be so by $\mathbb B_X$. \label{random}
\end{theorem}

The point is that Lemmas \ref{LOlemma} and \ref{projectivetruthisstable} also hold with ``Cohen" replaced by ``random". In light of this the proof of Theorem \ref{random} follows exactly the same lines as the proof of Theorem \ref{Cohen} with the proof of the case of one Cohen real replaced by the proof of Theorem \ref{Borel} in the special case of random forcing. The details are left to the reader.

\section{Sacks Forcing and its Relatives}

Having studied ccc forcing notions we now turn to non-ccc, proper forcing notions. In particular in this section we study Sacks forcing as well as some related forcing notions and prove a version of Theorem \ref{Cohen} for these. Recall that Sacks forcing $\mathbb S$ is the set of perfect $p \subseteq 2^{<\omega}$ ordered by inclusion. Here perfect means that for each $t \in p$ there is a $t' \supseteq t$ so that $t' \in p$ and $t'^\frown 0, t'^\frown 1 \in p$ where ${}^\frown$ denotes concatenation of strings. The main theorem of this section is the following.

\begin{theorem}
Let $\alpha$ be a countable ordinal, $n < \omega$ and $\Gamma \in \{\Sigma, \Pi\}$. If $\alpha$ is $\Gamma^1_n$-reflecting then $\mathbb S$ forces that $\check{\alpha}$ is $\Gamma^1_n$-reflecting.
    \label{sacks}
\end{theorem}

Note that unlike in the case of Cohen forcing we have not said anything about adding more than one Sacks real. We will address this later on in this section. First though note that in exactly the same manner as with Cohen forcing we have the following corollary.

\begin{corollary}
    In any generic extension of $L$ by a Sacks real $\sigma^1_n = (\sigma^1_n)^L$ and $\pi^1_n = (\pi^1_n)^L$ for all $n < \omega$. In particular the inequality $\sigma^1_{n+2} < \pi^1_{n+2}$ holds.
\end{corollary}

We will prove Theorem \ref{sacks} now. First recall some useful terminology about trees. Fix $p \in \mathbb S$. By $[p]$ we mean the set of branches of $p$. Note that $p \forces \dot{s} \in [\check{p}]$. Where $\dot{s}$ denotes the canonical name for the Sacks real. If $t \in p$ denote by $p_t$ the set of $u \in p$ compatible with $t$. Note that $p_t \in \mathbb S$ and strengthens $p$. A node $t\in p$ is a {\em split node}, if $t^\frown 0, t^\frown 1 \in p$. For a natural number $n < \omega$ we say that a split node $t \in p$ is an $n$-{\em splitting node} if there are $n$-many split nodes which are predecessors of $t$ in $p$. The set of $n$-splitting nodes is denoted ${\rm Split}_n(p)$. The unique $0$-splitting node is denoted the {\em root} of $p$. If $n < \omega$ and $p, q \in \mathbb S$ then we write $p \leq_n q$ if $p \subseteq q$ and ${\rm Split}_i(p) = {\rm Split}_i(q)$ for all $i \leq n$. A {\em fusion sequence} is a sequence $\{p_n\}_{n < \omega}$ of perfect trees so that for all $n < \omega$ we have $p_{n+1} \leq_n p_n$. The {\em fusion} of the sequence is the tree $p_\omega = \bigcap_{n < \omega} p_n$. We will need a standard fact about $\mathbb S$-names for reals, which we record now. Note that this is well known in the literature, but we include the proof, both in order to make the paper more self-contained and because the construction of $S$ will be relevant afterwards. 

\begin{lemma}
    There is an arithmetic set $S$ of countable $\mathbb S$-names so that if $p \in \mathbb S$, $\dot{x}$ is an $\mathbb S$-name and $p \forces$ ``$\dot{x}$ is a real" then there is a $q \leq p$ and a $\dot{y} \in S$ so that $q \forces \dot{y} = \dot{x}$.
\end{lemma}

Here by ``arithmetic set of $\mathbb S$-names" we mean that there is a set of reals definable by an arithmetic formula which code $\mathbb S$-names in a relatively concrete and simple way. 

\begin{proof}
    This is just a basic fusion argument, recording certain key steps. Suppose $p \in \mathbb S$, and $\dot{x}$ is forced by $p$ to name a real, which for concreteness we take as an element of $2^\omega$ (nothing important hinges on this -- the argument works just as well if $\dot{x}$ were forced to be a subset of some fixed $L_\alpha$). By strengthening if necessary we may assume that $p$ decides $\dot{x} (\check{0})$. Let $t \in p$ be the stem. Let $q_0$ strengthening $p_{t^\frown 0}$ decide $\dot{x}(\check{1})$ and $q_1$ strengthening $p_{t^\frown 1}$ decide $\dot{x}(\check{1})$. Let $p_1 = q_0 \cup q_1$. Note that $p_1 \leq_0 p_0$. Continuing in this way let $p_{n+1} \leq_n p_n$ so that for each $t \in {\rm Split}_{n}(p_{n+1})$ we have that $(p_{n+1})_t$ decides $\dot{x}(\check{n+1})$.  Let $q = \bigcap_n p_n$ be the fusion. Now let $y = \{\langle t, n, i\rangle \; | \; i \in 2 \; {\rm and} \; t \in {\rm Split}_n(q) \; {\rm and} \; q_t \forces \dot{x}(\check{n}) = \check{i}\} $. Clearly this codes an $\mathbb S$-name for a real in $2^\omega$ via the mapping 
\[y\mapsto \dot c_y = \Big\{ \big( q_t, (n,i) \big) : \langle t, n, i\rangle \in y\Big\}\]
where as before $q = \bigcup\{s : \exists n\,\exists i\, \langle s, n, i \rangle \in y\}$ and $q_t$ denotes the set of nodes in $q$ compatible with $t$. For such a name we have $q \forces \dot{c}_y = \dot{x}$. 

The idea is now to let $S$ be the set of all names such as this one. We write out the definition in more detail, since superficially it is a $\Sigma^1_1$ definition. We put $z \in S'$ if $z$ is a set of triples $\langle t, n, i\rangle \in 2^{<\omega} \times\omega\times 2$ such that $\langle t, 0, i\rangle \in z$ for precisely one $t$ and one $i$ and whenever $\langle t, n, i\rangle \in \sigma$, then $\langle t_0, n+1, i\rangle \in\sigma$ and $\langle t_1, n+1,i\rangle \in\sigma$ for precisely one $t_0$ extending $t^\frown 0$ and precisely one $t_1$ extending $t^\frown 1$.
If so, we may define $q^z = \bigcup\{s: \exists n\, \exists i\, \langle s, n,i \rangle \in z\} \in \mathbb{S}$ and define $\dot c_z$ just like $\dot c_y$, but using $q^z$ in place of $q$.
It is clear that $S'$ is arithmetical. Moreover, given a real $w$, one can decide arithmetically whether it is of the form $\dot c_z$ for some $z$, and  thus  $S$ is also arithmetical. Since the argument from the previous paragraph produces a real $y \in S'$, this shows that $S$ is as desired.
\end{proof}

From now on let us fix this set $S$ as in the lemma above. Note that in an easy way check names can be assumed to be in $S$. For simplicity if $q \in \mathbb S$ and $\dot{y} \in S$ so that $q$ is the union of the first coordinates of $\dot{y}$ as explained the proof above, or alternatively the unique real so that $\langle q, 0, i\rangle \in \dot{y}$, let us say that $q$ is the {\em witness} to $\dot{y} \in  S$. Note that the statement ``$q$ is a witness to $\dot{y} \in S$" is $\Delta^1_1$ with parameters $q$ and $\dot{y}$. 

The following lemma roughly says that all Sacks generic extensions (of a fixed ground model) agree about truth of statements involving check names -- in particular the truth of projective statements with ground model reals as parameters.
\begin{lemma}
    Let $a\in V$ and $\varphi (x)$ be a formula in the language of set theory. There is a condition $p \in \mathbb S$ forcing $\varphi (\check{a})$ if and only if every condition forces $\varphi(\check{a})$. \label{groundmodelequivalence}
\end{lemma}

\begin{proof}
    Fix $a$ and $\varphi(x)$ as in the statement of the lemma. If the lemma is false then there are conditions $p, q \in \mathbb S$ so that $p \forces \varphi(\check{a})$ and $q \forces \neg \varphi (\check{a})$. Fix such $p$ and $q$.  
    Now let $s = \bigcup G$ be Sacks over $V$ with $G \subseteq \mathbb S$ generic and $p \in G$. Work in $V[s]$. Let $s' \in [q]$ be defined as follows: for each $n < \omega$ and each $n^{\rm th}$-splitting node $t \in q$ let $s'(|t|) = 1$ if and only if $l^\frown 1 \subseteq s$ where $l \in p$ is the unique $n^{\rm th}$-splitting node extended by $s$. In other, less rigorous words, we defined $s'$ by reading off of $s$ whether we went left or right at the splitting nodes of $p$ and doing the analogous move on $q$. Note that $s'$ is new -- since otherwise we could find $s$ in the ground model -- and hence by the minimality of Sacks forcing, $s'$ is a Sacks generic real over $V$ and $V[s] = V[s']$, see Sacks \cite[Theorem 1.12]{Sa71}. However this is a contradiction since $s'$ comes from a generic containing $q$ and hence on the one hand we have $V[s] \models \varphi (a)$ and on the other hand $V[s] = V[s'] \models \neg \varphi (a)$.
    \end{proof}

The next two lemmas are intended to show that forcing a projective statement with Sacks forcing has the same complexity as the statment itself. First we need to worry about the case where the statement is arithmetic. 
\begin{lemma}\label{arithforcingsacks}
Let $\psi(x)$ be arithmetic. For any $\dot{x} \in S$ and any condition $p \in \mathbb S$ the statement $p \forces \psi (\dot{x})$ is $\Delta^1_1$ with parameters $p$ and $\dot{x}$ (with the later being treated by the real coding it, not its denotation in some generic extension). 
\end{lemma}

\begin{proof}
We consider the case that $\psi(x)$ is of the form $n \in x$ for some natural number $n < \omega$ a simple induction on complexity of $n$ shows how to finish the proof given this case. Therefore we want to show that the statement that $p \forces \check{n} \in \dot{x}$ is $\Delta^1_1$ in $p$ and $\dot{x}$. Recall that $\dot{x}$ consists of triples of the form $\langle q_t, n, i\rangle$ with the union of the first coordinates a member of $\mathbb S$ and $i < 2$. Now $p$ forces that $n$ is in $\dot{x}$ just in case $p \cap q$ is a condition for $q$ the witness to $\dot{x} \in S$ and for every $\langle q_t, n, i\rangle \in \dot{x}$ we have $i = 1$ if and only if $t \supseteq p$.  Since this statement is clearly $\Delta^1_1$ we are done.
\end{proof}

\begin{lemma}
    Let $n < \omega$, $\Gamma \in \{\Pi, \Sigma\}$, $r \in 2^\omega \cap V$ and let $\varphi$ be a $\Gamma^1_n$-formula. The statement ``$ \forces \varphi (\check{r})$" is also $\Gamma^1_n$. \label{forcingissamecomplexity}
\end{lemma}

\begin{proof}
 Fix $n < \omega$, $\Gamma \in \{\Pi, \Sigma\}$ $r$ and $\varphi$ a $\Gamma^1_n$ formula as in the statement of the lemma. There are two cases depending on whether $\Gamma$ is $\Pi$ or $\Sigma$. We take them one at a time. However, the proof in both cases will be by induction on $n$ and we will need to interweave this induction to get the next step so let us explain what it is first. First suppose $\Gamma = \Sigma$ and let $\varphi (r)$ be of the form $\exists X_0 \forall X_1 \exists X_2 \hdots \psi (r, X_0, X_1, X_2, \hdots)$ with $\psi$ arithmetic. We claim that $\forces \varphi (r)$ is equivalent to $\exists p_0 \exists \dot{X}_0 \in S \forall p_1 \leq p_0 \forall \dot{X}_1 \exists p_2 \leq p_1 \exists \dot{X}_2 \hdots p_{n-1} \forces \psi(\check{r}, \dot{X}_0, \dot{X}_1, \dot{X}_2, \hdots)$. In light of Lemma \ref{arithforcingsacks} this later statement is $\Sigma^1_n$ so if we can show this we're done. Similarly if $\Gamma = \Pi$ then let $\varphi (r)$ be of the form $\forall X_0 \exists X_1 \forall X_2 \hdots \psi (r, X_0, X_1, X_2, \hdots)$ with $\psi$ arithmetic. Again here we claim that $\forces \varphi (r)$ is equivalent to $\forall p_0 \forall \dot{X}_0 \in S \exists p_1 \leq p_0 \exists \dot{X}_1 \forall p_2 \leq p_1 \forall \dot{X}_2 \hdots p_{n-1} \forces \psi(\check{r}, \dot{X}_0, \dot{X}_1, \dot{X}_2, \hdots)$. Assume now that both statements are true for some fixed $n$ - the base case follows from Lemma \ref{arithforcingsacks}.\smallskip

\noindent \underline{Case 1}: $\Gamma = \Sigma$. We assume that $\varphi (r) \in \Sigma^1_{n+1}(r)$ is a formula having the form $\exists X_0 \forall X_1 \exists X_2 \hdots \psi (r, X_0, X_1, X_2, \hdots)$ with $\psi$ arithmetical. Again, we claim that $\forces \varphi (r)$ is equivalent to $$\exists p_0 \exists \dot{X}_0 \in S \forall p_1 \leq p_0 \forall \dot{X}_1 \exists p_2 \leq p_1 \exists \dot{X}_2 \hdots p_{n} \forces \psi(\check{r}, \dot{X}_0, \dot{X}_1, \dot{X}_2, \hdots).$$ 
Clearly if $\forces \varphi(\check{r})$ then the displayed statement holds. If, however, $\not\forces \varphi(\check{r})$ then, by Lemma \ref{groundmodelequivalence} we have that $\forces \neg \varphi (\check{r})$ which means that, when combined with the inductive hypothesis, we get that 
\[\neg \big[\exists p_0 \exists \dot{X}_0 \in S \forall p_1 \leq p_0 \forall \dot{X}_1 \exists p_2 \leq p_1 \exists \dot{X}_2 \hdots p_{n} \forces \psi(\check{r}, \dot{X}_0, \dot{X}_1, \dot{X}_2, \hdots)\big],\] 
as needed.\smallskip

\noindent \underline{Case 2}: $\Gamma = \Pi$. This is almost the same as Case 1 but easier. Again we assume that $\varphi (r) \in \Pi^1_{n+1}(r)$ is of the form $\forall X_0 \exists X_1 \forall X_2 \hdots \psi (r, X_0, X_1, X_2, \hdots)$ with $\psi$ arithmetic. We claim that $\forces \varphi (r)$ is equivalent to 
\[\forall p_0 \forall \dot{X}_0 \in S \exists p_1 \leq p_0 \exists \dot{X}_1 \forall p_2 \leq p_1 \forall \dot{X}_2 \hdots p_{n} \forces \psi(\check{r}, \dot{X}_0, \dot{X}_1, \dot{X}_2, \hdots).\] 
Once more the forward direction is obvious. The backward direction now though follows more readily: if every $p_0$ forces every name in $S$ to have the property $\exists X_1 \forall X_2 \hdots \psi (r, X_0, X_1, X_2, \hdots)$ then clearly the maximal condition forces that every name has this property as needed.
\end{proof}

    Given these three lemmas we can now prove Theorem \ref{sacks}.

    \begin{proof}[Proof of Theorem \ref{sacks}]
        Fix $n < \omega$, $\Gamma \in \{\Pi, \Sigma\}$ and $\alpha < \omega_1$ which is $\Gamma^1_n$-reflecting. Let $\varphi (x)$ be a $\Gamma^1_n$-forumla, $\beta < \alpha$ and assume that some condition $p \in \mathbb S$ forces that $L_\alpha \models \varphi (\beta)$. By Lemma \ref{groundmodelequivalence} we actually get that {\em every condition} forces this, in particular we have that $\forces_\mathbb S L_\alpha \check{} \models \varphi (\check{\beta})$. By Lemma \ref{forcingissamecomplexity} this is itself a $\Gamma^1_n$-statement (with only parameter $\beta$) so there is some $\gamma \in (\beta, \alpha)$ to which it reflects as needed.  
    \end{proof}

    The above lemmas and hence theorem relied very little on the properties of Sacks forcing per se and hold almost verbatim for any Borel tree forcing with the right kind of fusion and the properties described in Lemmas \ref{groundmodelequivalence} and \ref{forcingissamecomplexity}. In particular we leave it to the reader to check that the following also hold.

    \begin{theorem}\label{lavermiller}
        Let $\alpha$ be a countable ordinal, $n < \omega$ and $\Gamma \in \{\Sigma, \Pi\}$. If $\alpha$ is $\Gamma^1_n$-reflecting then $\mathbb P$ forces that $\check{\alpha}$ is $\Gamma^1_n$-reflecting for $\P$ either Laver or Miller forcing. 
    \end{theorem}

    Of course Laver or Miller here could be replaced by several other well known forcing notions as well.

    Let us briefly remark on the issue in trying to extend the above to countable support iterations. Unlike the case of Cohen or random forcing, in iterations the well orderedness of the index set is extremely important as the partial orders are defined inductively. This means that first of all it is not clear how to reduce the number of steps needed and second of all it is not clear that any $L_\alpha$ can even define iterations of length longer than some ordinal $\beta$ greater than the supremum of the ordinals it can define. In short there is no clear path to turning statements of the form $\forces_\P \varphi$ into projective statements when $\varphi$ is projective if $\P$ is a countable support iteration of long length. Nevertheless it would be extremely surprsing if countable support iterations of Sacks forcing (say) could change reflectingness of ordinals when the single step does not, and we conjecture that they do not.

\section{Collapsing} \label{SectCollapse}
Given the results presented so far, the reader might wonder if \textit{any} partial order $\P$ can increase the value of $\sigma^1_n$ or $\pi^1_n$. The purpose of this section is to give a positive answer. The example will essentially be L\'evy's partial order for collapsing $\omega_1$. The ordinals $\sigma^1_n$ or $\pi^1_n$ behave very strangely in extensions of $L$ by this partial order: we will see that they are strictly larger than the corresponding ordinals in $L$, yet smaller than $\omega_1^L$.

For simplicity let us write $\P=\P^L$ for the forcing notion consisting of finite sequences of reals ordered by end extension (as computed in $L$). This forcing clearly collapses the reals of the ground model to be countable (and in fact is forcing equivalent to ${\rm Col}(\omega, 2^{\aleph_0})$ if $V=L$. Note that in any model of set theory this $\P$ is a $\Sigma^1_2$ set of reals. The main theorem of this section is the following:

\begin{theorem}    \label{collapse}
Let $3 \leq n,m < \omega$ and let $G\subset \P$ be $L$-generic. Then,
\[(\sigma^1_m)^L < (\pi^1_m)^L < (\sigma^1_n)^{L[G]} < (\pi^1_n)^{L[G]} < \omega_1^L < (\omega_1)^{L[G]} = \omega_2^L. \]
\end{theorem}
In other words, every old $\sigma^1_n$ and $\pi^1_n$ ($n > 2$) is below every new such one but all of these are still less than the $\omega_1$ of $L$. Moreover, the projective reflecting ordinals are ordered as in $L$.

Below, we identify an $L$-generic $G\subset\P$ with the single real $x$ coding the generic $\omega$-sequence of $L$-reals. 
We work towards Theorem \ref{collapse} beginning with the following lemma. 

\begin{lemma}\label{LemmaGenericPi12}
    Let $\vec{x} = \langle x_n\; |\; n < \omega\rangle$. The statement ``$\vec{x}$ is $\P$-generic over $L$" is $\Pi^1_2$ with parameter $\vec{x}$. 
\end{lemma}

Before beginning the proof of this lemma note that if there is a $\P$-generic over $L$ then $L_{\omega_1}$ is countable. This will come up several times in the proof as well as in the rest of this section. 
\begin{proof}
   Note the lemma could have also been written as ``the set of $\P$-generic $\vec{x}$ is $\Pi^1_2$ in $\mathbb R^\omega$". If $L_{\omega_1}$ is uncountable then trivially there is no $\P$-generic over $L$ hence ``$\vec{x}$ is $\P$-generic" is vapidly $\Pi^1_2$ (and in fact much lower complexity). From now on therefore assume that $L_{\omega_1}$ is countable. 
   
   Since $\P \subseteq L_{\omega_1}$ this means that $\P$ itself is countable and hence so is any dense subset. Given this we can say that $\vec{x} = \langle x_n \; | \; n < \omega \rangle \in \mathbb R^\omega$ is $\P$-generic over $L$ if and only if for every dense subset $D \in L$ of $\P$ there is an $l < \omega$ with $\langle x_i \; | \; i < l\rangle \in D$. Note that every dense subset of $\P$ in $L$ requires every real to be constructible while for each constructible real $y$ there is a dense subset of conditions such that $y$ is included in the sequence thus the above implies that $\vec{x}$ is a sequence containing all and only the constructible reals. It remains therefore to show that this statement is in fact $\Pi^1_2$.

   To see this, observe first that the statement ``$D$ is dense" is naively $\Pi^1_2$ with $D$ as a parameter once $D$ is countable since we need to say that for each $y \in \P$ there is a $q \in D$ so that $q \supseteq p$. The complexity comes from the fact that the definition of $\P$ is $\Sigma^1_2$ as it is the set of finite sequences of constructible reals. It follows that for any transitive structure $M$ of a sufficient fragment of $\ZFC$ will be correct that ``$D$ is dense" assuming $M$ knows that $D$ is countable. Moreover, observe that if $\vec{x} \in M$ then $M$ knows that $\P$ is countable (since $\vec{x}$ is a well order of the elements of $\P$ in order type $\omega$). Therefore we get that $D$ is dense if and only if for all $M$ if $M$ is transitive and $\vec{x}, D \in M$ and $D \in L$ and $M \models$`` $D$ is dense" then there is an $l < \omega$ so that $\langle x_i \; | \; i < l\rangle \in D$. This latter statement is $\Pi^1_2$ in $\vec{x}$ as it can be written formally as follows: $$\forall M \forall D \Big[M \, {\rm transitive} \land D, \vec{x} \in M \land D \in L \land M \models \neg D \, {\rm dense} \to \exists l < \omega \, \langle x_i \; | \; i < l\rangle \in D\Big]$$ 
   or even more pedantically as $$\forall M \forall D\Big[ M \; {\rm not \, transitive} \lor \vec{x} \notin M \lor D \notin M \lor$$$$ D \notin L \lor M \models \neg D \, {\rm dense} \lor \exists l < \omega \, \langle x_i \; | \; i < l\rangle \in D\Big]$$ This therefore completes the proof.
\end{proof}

The following lemma establishes a weak form of some of the inequalities in Theorem \ref{collapse}. Although we could have omitted it and worked directly towards the proof of the theorem, we find it instructional to include the argument, which will be refined later.
\begin{lemma}\label{LemmaKillingReflection}
Assume $V = L$ and let $3\leq n < \omega$  and $\gamma \in \{\sigma,\pi\}$.  Then, we have $\forces_\P (\gamma^1_n)^L < \gamma^1_n$.
\end{lemma}
\begin{proof}
By Lemma \ref{dontgodown} we know that $(\gamma^1_n)^L \leq \gamma^1_n$, so to prove the lemma it suffices to show that $(\gamma^1_n)^L \neq \gamma^1_n$.
The key is to use Lemma \ref{LemmaGenericPi12}. 

Observe that, since all reals of $L$ belong to $L_{\omega_1^L}$, the relation $x = \gamma^1_n$ is first-order expressible over $L_{\omega_1^L}$, say by a formula $\phi_{\gamma^1_n}(x)$. Working in $L[G]$, where $G\subset\P$ is $L$-generic, we have
\begin{align}\label{eqDefSigma1nLG}
x = (\sigma^1_n)^L \leftrightarrow
\exists M\,\exists y\subset\P\, \Big(&\text{$y$ is $L$-generic} \wedge y \in M \wedge M \, {\rm transitive} \\
& \wedge M\models \mathsf{ZC} \wedge (L_{\omega_1^L})^M \models \phi_{\sigma^1_n}(x)\Big). \nonumber
\end{align}
To see this (still working in $L[G]$), let $M$ be the transitive colapse of a countable elementary substructure of $L_{\aleph_\omega}^M[G]$ such that $G \in M$ and $L_{\omega_1^L}\subset M$. Then $M$ correctly determines that $G$ is $L$-generic and it correctly identifies $L_{\omega_1^L}$, so it correctly assesses that $(\sigma^1_n)^L$ is $(\sigma^1_n)^L$.

Conversely, if the formula on the right-hand side of \eqref{eqDefSigma1nLG} holds, then the fact that $y \subset \P$ is $L$-generic and $y \in M$ implies that $L_{\omega_1^L}\subset M$, so we really must have $x = (\sigma^1_n)^L$. This proves \eqref{eqDefSigma1nLG}, which is a $\Sigma^1_n$ formula $\psi$ such that $L_{\sigma^1_n} \models  \psi$ and $L_{\alpha}\not\models\psi$ for any other $\alpha$. This means that $(\sigma^1_n)^L < \sigma^1_n$. 

For $\pi^1_n$, we take 
\begin{align*}
x = (\pi^1_n)^L \leftrightarrow
\forall M\,\forall y\subset\P\, \Big(&\text{$y$ is $L$-generic} \wedge y \in M \wedge M \, {\rm transitive} \\
&\wedge M\models \mathsf{ZC} \wedge (L_{\omega_1^L})^M \to \phi_{\sigma^1_n}(x)\big)
\end{align*}
and argue similarly.
\end{proof}

Note that in the above in fact the $\psi$ was not just $\Sigma^1_n$ (or $\Pi^1_n$ respectively) but in fact $\Sigma^1_3$ (or respectively $\Pi^1_3$) thus the above shows that for all $n >2$ the ordinals $(\sigma^1_3)^L$ and $(\pi^1_3)^L$ are no longer even $\Sigma^1_3$ or $\Pi^1_3$ reflecting after forcing with $\P$. This will be elaborated on later. 

The proof of Lemma \ref{LemmaKillingReflection} yields the following general fact:
\begin{corollary}\label{CorollaryOmega2L}
Suppose that $\omega_2^L < \omega_1$. Then, for all $n$ with $3 \leq n < \omega$, we have 
\begin{align*}
(\sigma^1_n)^L < \sigma^1_n \qquad\text{ and }\qquad (\pi^1_n)^L < \pi^1_n.
\end{align*}
\end{corollary}
\begin{proof}
This follows from the argument of  Lemma \ref{LemmaKillingReflection}: if $\omega_2^L < \omega_1$, then $\P^L$ only has countably many dense subsets in $L$, so there is some $G\subset\P^L$ which is $L$-generic. 
But then as in Lemma \ref{LemmaKillingReflection} there is a $\Sigma^1_n$ formula $\varphi$ such that $L_{\sigma^1_n} \models \varphi$
 and $L_\alpha \not\models\varphi$ 
 for all smaller $\alpha$, 
 and similarly for $\pi^1_n$.
\end{proof}

According to the following lemma, the values of $(\sigma^1_n)^{L[G]}$ and $(\pi^1_n)^{L[G]}$ are bounded by $\omega_1^L$. The technique of the proof will recur and might be useful in other contexts:
\begin{lemma}\label{LemmaReflectionCountable}
Assume $V = L$ and let $3\leq n < \omega$  and $\gamma \in \{\sigma,\pi\}$.  Then, $\forces_\P \gamma^1_n < \omega_1^L$.
\end{lemma}
\begin{proof}
Work in any model of set theory for now and let $\gamma < \pi^1_n$ (the argument for $\sigma^1_n$ is similar). Since $\gamma < \pi^1_n$, there is some $\Pi^1_n$ formula $\psi$ and ordinals $\gamma_1 < \gamma_2 < \dots < \gamma_k < \gamma$ such that 
$L_\gamma\models\psi(\gamma_1, \hdots, \gamma_k)$,
but $L_{\bar\gamma} \not\models \psi(\gamma_1, \hdots, \gamma_k)$ for all $\bar\gamma < \gamma$ with $\gamma_k < \bar\gamma$. 
Inductively, we can construct a tree $T_\gamma$ with root $(\psi, \gamma_1, \hdots, \gamma_k)$ and where to each node 
$s = (\psi_s, \gamma^s_1, \hdots, \gamma^s_{k_s})$
are attached $k_s$ different successors $s_1, s_2, \hdots, s_{k_s}$, where $s_i$ is a tuple of the form $(\theta, \delta_1, \hdots, \delta_l)$ such that the following hold:
\begin{enumerate}
\item $\delta_1 <  \hdots < \delta_l < \gamma^s_i$;
\item $L_{\gamma^s_i} \models \theta(\delta_1, \hdots, \delta_l)$;
\item $L_{\bar\gamma} \not\models \theta(\delta_1, \hdots, \delta_l)$ whenever $\delta_l < \bar\gamma < \gamma^s_i$.
\end{enumerate}
This way, each node in the tree can be seen as naming an ordinal $\leq\gamma$, with the ordinals named decreasing along each branch of $T_\gamma$. Hence, $T_\gamma$ is a finitely branching tree with no infinite branch and is hence finite, by K\"onig's Lemma. Terminal nodes in $T_\gamma$ consist each of a single $\Pi^1_n$ formula and no ordinal parameters. Each of these formulas $\theta$ can be thought of as naming the least ordinal $\delta$ such that $L_\delta\models\theta$. 

Let $T^-_\gamma$ be the result of deleting all ordinals from $T_\gamma$; thus, $T^-_\gamma$ is simply a finite tree of $\Pi^1_n$ formulas. By induction from the leaves down to the root, we see that $T_\gamma$ is determined uniquely by $T^-_\gamma$. We have seen that every ordinal $\gamma<\pi^1_n$ can be assigned such a tree $T^-_\gamma$, and each finite tree $T$ of $\Pi^1_n$ formulas corresponds to at most one ordinal $\gamma<\pi^1_n$; if so, we say that $T$ \textit{denotes} $\gamma$.

Now, let $G\subset \P$ be $L$-generic. If $\gamma < (\pi^1_n)^{L[G]}$, then there is some finite tree $T$ of $\Pi^1_n$ formulas which denotes $\gamma$ in $L[G]$. By the weak homogeneity of $\P$, we have $\forces_\P$ ``$T$ denotes $\gamma$.'' Let 
\[S = \min\Big\{ \mathsf{Ord} \setminus \big\{\gamma: \exists T\, (\text{$\forces_\P$ ``$T$ denotes $\gamma$'')}\big\} \Big\}.\]
Then, $S$ is countable in $L$ and $\Vdash_\P S = \pi^1_n$, so $(\pi^1_n)^{L[G]} < \omega_1^L$. As we mentioned before, the argument for $\sigma^1_n$ is similar.
\end{proof}

Unfortunately, Lemma \ref{LemmaReflectionCountable} cannot be extended to arbitrary outer models in the spirit of Corollary \ref{CorollaryOmega2L}:

\begin{proposition}
Suppose $0^\sharp$ exists. Then, $\omega_1^L < \sigma^1_3$ and $\omega_1^L < \pi^1_3$.
\end{proposition}
\begin{proof}
Recall a theorem of Solovay \cite{So67} according to which $0^\sharp$ is $\Delta^1_3$-definable. If $0^\sharp$ exists, every ordinal smaller than the least $L$-indiscernible, $\zeta_0$ is parameter-free definable in $L$ and hence recursive in $0^\sharp$. Let $e_\zeta$ be an element of $\mathcal{O}^{0^\sharp}$ of order-type $\zeta$ from $0^\sharp$, where $\mathcal{O}^x$ denotes Kleene's system of ordinal notatinos relativized to $x$. If $x$ is any real number, then $e_\zeta$ could also be regarded as an $x$-ordinal notation by replacing all oracle calls to $0^\sharp$ by oracle calls to $x$. We write $e_\zeta^x$ to emphasize this fact (note $e_\zeta^x$ might not result in an element of $\mathcal{O}^x$, since it might code an illfounded $x$-recursive order when $x\neq 0^\sharp$).
Then, $\zeta$ is the only ordinal such that
\[L_\zeta \models \exists x\, (x = 0^\sharp \wedge \mathsf{Ord} = |e_\zeta^x|),\]
where $|e_\zeta^x|$ denotes the order-type of the notation $e_\zeta^x \in\mathcal{O}^x$ and both the quantification over $x$ and the formula $x = 0^\sharp$ are expressed in second-order logic. Similarly, $\zeta$ is the only ordinal such that
\[L_\zeta \models \forall x\, (x = 0^\sharp \to \mathsf{Ord} = |e_\zeta^x|).\]
These are respectively $\Sigma^1_3$ and $\Pi^1_3$ formulas, so the result follows. 
\end{proof}

\begin{lemma}\label{LemmaSigmaLG}
Assume $V = L$ and let $3\leq n < \omega$. Then, $\forces_\P (\pi^1_n)^L < \sigma^1_3$.
\end{lemma}
\begin{proof}
The idea of the proof is to extend the argument of Lemma \ref{LemmaKillingReflection} by incorporating the ideas from the proof of Lemma \ref{LemmaReflectionCountable}.

Let $\gamma < (\pi^1_n)^L$ and let $T_\gamma^-$ be as in Lemma \ref{LemmaReflectionCountable}. Recall that $T_\gamma^-$ is a finite tree of $\Pi^1_n$ formulas.
The fact that $T_\gamma^-$ denotes $\gamma$ (in the sense of Lemma \ref{LemmaReflectionCountable}) is first-order expressible over $L_{\omega_1^L}$, since it is a Boolean combination of $\Pi^1_n$ properties of $\gamma$ and a set of smaller ordinals which are determined uniquely from $T_\gamma^-$. We now argue similarly to the proof of Lemma \ref{LemmaKillingReflection} by observing that the formula $x = \gamma$ is equivalent, in $L[G]$, to
\begin{align}\label{eqDefSigma1nLGStr}
\exists M\,\exists y\subset\P\, \Big(&\text{$y$ is $L$-generic} \wedge y \in M \wedge M \, {\rm transitive} \\
& \wedge M\models \mathsf{ZC} \wedge (L_{\omega_1^L})^M \models \text{``$T^-_\gamma$ denotes $x$''}\Big).\nonumber
\end{align}
Observe that \eqref{eqDefSigma1nLGStr} is a $\Sigma^1_3$ formula, by Lemma \ref{LemmaGenericPi12}. The fact that \eqref{eqDefSigma1nLGStr} is equivalent to $x = \gamma$ is proved just like \eqref{eqDefSigma1nLG} in the proof of Lemma \ref{LemmaKillingReflection}. This shows that $(\pi^1_n)^L \leq (\sigma^1_3)^{L[G]}$, but the strict inequality follows from Lemma \ref{LemmaKillingReflection} or, alternatively by increasing $n$.
\end{proof}

The main ingredient missing towards the proof of Theorem \ref{collapse} is the fact that $\sigma^1_n < \pi^1_n$ holds in $L[G]$, which will take place in the remainder of this section. Towards establising this inequality, we introduce two auxiliary ordinals. 

\begin{definition}
We define the ordinal $\sigma_{\Sigma_n,\omega_2}$ to be the least ordinal $\sigma$ such that whenever $\phi \in \Sigma_n$ and $\gamma_1 < \gamma_2 < \dots < \gamma_k < \sigma$ are ordinals, if $L_{\omega_2^L}\models \phi(\sigma, \gamma_1, \hdots, \gamma_k)$, then there is $\bar\sigma$ such that $\gamma_k <\bar\sigma < \sigma$ and $L_{\omega_2^L}\models \phi(\bar\sigma, \gamma_1, \hdots, \gamma_k)$. 

The ordinal $\sigma_{\Pi_n,\omega_2}$ is defined analogously.
\end{definition}

Let us remark that if $M$ is the transitive collapse of a countable elementary substructure of some large $H(\kappa)$, then the ordinals $\sigma_{\Sigma_n,\omega_2}^M$ and $\sigma_{\Pi_n,\omega_2}^M$ satisfy the defining properties of $\sigma_{\Sigma_n,\omega_2}$ and $\sigma_{\Pi_n,\omega_2}$; it follows that -- despite the occurrence of $\omega_2$ in their definition -- these ordinals are countable.

\begin{lemma}\label{LemmaSigmaOmega2}
Suppose that $2 \leq n < \omega$ and $G\subset \P$ is $L$-generic. Then, $(\sigma^1_{n+1})^{L[G]} = \sigma_{\Sigma_n,\omega_2}^L$ and $(\pi^1_{n+1})^{L[G]} = \sigma_{\Pi_n,\omega_2}^L$.
\end{lemma}
\proof
Let us show that $(\sigma^1_{n+1})^{L[G]} = \sigma_{\Sigma_n,\omega_2}^L$; the proof that $(\pi^1_{n+1})^{L[G]} = \sigma_{\Pi_n,\omega_2}^L$ is analogous.

The proof that  $\sigma_{\Sigma_n,\omega_2}^L \leq (\sigma^1_{n+1})^{L[G]}$ is easy: every $\Sigma_2$ formula about $L_{\omega_2}$ can be uniformly translated into into a $\Sigma_2$ formula about $L_{\omega_1}^{L[G]}$ via $\varphi\mapsto \varphi^L$. In turn, every $\Sigma_2$ formula about $L_{\omega_1}^{L[G]}$ with parameters $\sigma, \vec\gamma$ can be uniformly translated into a $\Sigma^1_3$ formula over $L_\sigma$ with parameters $\gamma$.

To show that $(\sigma^1_{n+1})^{L[G]} \leq \sigma_{\Sigma_n,\omega_2}^L$, it suffices to prove that $\P$ forces that  $\sigma_{\Sigma_n,\omega_2}^L$ is $\Sigma^1_{n+1}$-reflecting. As in the proof of Theorem \ref{Cohen}, we assume without loss of generality that $n =2$ and suppose towards a contradiction that there is some $\Sigma^1_3$ formula $\varphi$ and parameter $\beta < \sigma_{\Sigma_n,\omega_2}^L =: \sigma$ such that 
\[\forces_\P \text{``$L_\sigma \models \varphi(\beta)$'' $\wedge$ ``$\forall \bar\sigma\, \big(\beta < \bar\sigma < \sigma \to L_{\bar\sigma}\not\models\varphi(\beta)\big)$.''}\]
For definiteness, write $\varphi$ as
a formula of the form $\exists X\forall Y\exists Z \psi(X,Y,Z,x)$. By the Gandy basis theorem, the following are equivalent for all countable $\alpha$:
\begin{enumerate}
\item $L_\alpha \models \exists X\forall Y\exists Z \psi(X,Y,Z,\beta)$; and
\item there is $X\subset\alpha$ such that for all $Y_0\subset\alpha, Y_1 \subset\mathbb{N}$, if $Y_1$ is a structure isomorphic to $(L_{\alpha'}, X, Y_0)$, where $\alpha < \alpha'$ and $\alpha'$ is admissible, then there is $Z \leq_T Y_1$ such that $L_\alpha\models \psi(X,Y,Z,\beta)$.
\end{enumerate}
Thus, if $\alpha<\omega_1^{L[G]} = \omega_2^L$, then $L_\alpha \models \exists X\forall Y\exists Z \psi(X,Y,Z,\beta)$ can be expressd in $L[G]$ as a $\Sigma_2$ formula over $L_{\omega_2^L}$, uniformly in $\alpha$.
Arguing as in Lemma \ref{LOlemma}, we see that the formula $\forces_\P$ ``$L_\sigma \models \varphi(\beta)$'' is thus equivalent to a formula $\theta(\sigma)$ of the form
\begin{align*}
\exists \dot X \in \P{\rm -}\mathsf{NAME}\,\forall \dot Y_0, \dot Y_1 \in \P{\rm -}\mathsf{NAME}
&\Big(\forces_\P Y_1 \cong (L_{\sigma'}, \dot X, \dot Y_0) \wedge \sigma < \sigma' \wedge L_{\sigma'}\models\mathsf{KP}\\
&\to\, \forces_\P \exists \dot Z\leq_T \dot Y_1\, L_\sigma\models \psi(\dot X, \dot Y_0, \dot Z, \beta)\Big).
\end{align*}
As in the proof of Theorem \ref{Cohen}, we see that the quantification over $\P$-names can be replaced by quantification over $L_{\omega_2^L}$, so $\theta(\sigma)$ can be expressed as a $\Sigma_2$ formula over $L_{\omega_2^L}$, uniformly for all $\sigma$, so we obtain a contradiction to the definition of $\sigma_{\Sigma_2,\omega_2}$.
\endproof

\begin{lemma}\label{LemmaSigmaPiInLG}
Assume $V = L$ and let $3\leq n < \omega$. Then, $\forces_\P \sigma^1_n < \pi^1_n$.
\end{lemma}
\begin{proof}
By Lemma \ref{LemmaSigmaOmega2}, it suffices to work in $L$ and show that if $2\leq n < \omega$, then $\sigma_{\Sigma_n,\omega_2} < \sigma_{\Pi_n,\omega_2}$. Thus we work in $L$. We have noted that $\sigma_{\Sigma_n,\omega_2}$ and $\sigma_{\Pi_n,\omega_2}$ are countable. 

Let $\alpha$ be a countable ordinal and denote by $\delta_{n,\omega_2}(\alpha)$ the supremum of order-types of $\Delta_n^{L_{\omega_2}}(\{\alpha\})$ wellorderings of $\mathbb{N}$. These are the wellorderings $R(x,y)$ of $\mathbb{N}$ for which there exist $\Sigma_n$ formulas $\phi_0,\phi_1$ such that for all $x,y\in\mathbb{N}$, we have
\begin{align}\label{DefRelR}
R(x,y)
&\leftrightarrow L_{\omega_2} \models \phi_0(x,y,\alpha)\\
&\leftrightarrow L_{\omega_2} \not\models \phi_1(x,y,\alpha).\nonumber
\end{align}
Note that the class of ordinals given this way is closed downwards, so that $\delta_{n,\omega_2}(\alpha)$ is also the first ordinal not of this form.

\begin{claim}
If $\alpha < \sigma_{\Sigma_n,\omega_2}$, then $\delta_{n,\omega_2}(\alpha)\leq \sigma_{\Sigma_n,\omega_2}$. Similarly, if $\alpha < \sigma_{\Pi_n,\omega_2}$, then $\delta_{n,\omega_2}(\alpha)\leq \sigma_{\Pi_n,\omega_2}$.
\end{claim}
\proof
This argument follows the proof of Aczel-Richter \cite[Theorem 11.1]{AcRi74}. We include the proof for $\sigma_{\Sigma_n,\omega_2}$; the proof for $\sigma_{\Pi_n,\omega_2}$ is similar.
Let $\delta < \delta_{n,\omega_2}(\alpha)$, say, given by a relation $R$ as in \eqref{DefRelR}.
Let $\theta(\alpha,\sigma)$ be a $\Sigma_n$ formula which over $L_{\omega_2}$  asserts the existence (first-order) of a bijection $f:\mathbb{N}\to\sigma$ such that for all $x,y,\in\mathbb{N}$, we have $f(x) < f(y)$ if and only if $R(x,y)$ holds. Then, clearly $\sigma$ is the least ordinal such that $L_{\omega_2}\models \theta(\alpha,\sigma)$.
\endproof

It follows that each of $\sigma_{\Sigma_n,\omega_2}$ and $\sigma_{\Pi_n,\omega_2}$ is either of the form $\delta_{n,\omega_2}(\alpha)$, or a limit of these ordinals.

We now claim that, letting $\delta:=\delta_{n,\omega_2}(\alpha)$, we have 
$L_\delta \prec_{\Sigma_n} L_{\omega_2}$ for every $\alpha < \omega_1$,
i.e., we have: if $\phi$ is a $\Sigma_n$ formula and $\gamma_1 < \hdots < \gamma_k <\delta$ are such that $L_{\omega_2} \models \phi(\gamma_1, \hdots, \gamma_k)$, then $L_\delta\models \phi(\gamma_1, \hdots, \gamma_k)$. To prove this, we assume for simplicity that there is only one parameter $\gamma$, so that $L_{\omega_2} \models \exists x\, \phi(\gamma,x)$, where $\phi \in \Pi_{n-1}$. We let $R_\gamma$ be a $\Delta_n^{L_{\omega_2}}(\{\alpha\})$ wellordering of $\mathbb{N}$ isomorphic to $\gamma$ and write $R_\gamma^x$ for the result of replacing $\alpha$ in the definition of $R_\gamma$ by $x$. Thus, $R_\gamma^\alpha = R_\gamma$ and $R_\gamma^\beta \subset R_\gamma^\alpha$ whenever $\beta<\alpha$ and $L_\beta\prec_{\Sigma_{n-1}}L_\alpha$.
Define a wellordering $R$ of $\mathbb{N}$ by putting $x <_R y$ if and only if there exist:
\begin{enumerate}
\item an ordinal $\xi < \omega_1$, 
\item a wellordering $Q_\alpha$ of $\mathbb{N}$ isomorphic to $\alpha$,
\item an ordinal $\zeta$ isomorphic to $R^{\textsc{lth}(Q_\alpha)}_\gamma$, where $\textsc{lth}(Q_\alpha)$ denotes the length of $\textsc{lth}(Q_\alpha)$,
\item a wellordering $Q_\xi$ of $\mathbb{N}$ isomorphic to $\xi$,
\end{enumerate}
such that $(Q_\alpha, \zeta, \xi, Q_\xi)$ are lexicographically ${<}_L$-least such that $L_{\omega_2}\models \phi(\zeta, \xi)$, and $x <_{Q_\xi} y$.

Observe that the quadruple $(Q_\alpha, \zeta, \xi, Q_\xi)$ is uniquely determined, $\xi$ is a witness to the fact that $L_{\omega_2}\models \exists x\, \phi(\zeta,x)$ and $R\cong Q_\xi \cong \xi$. However, the definition just given of $R$ is $\Sigma_n^{L_{\omega_2}}(\{\alpha\})$. Moreover, since the quadruple $(Q_\alpha, \zeta, \xi, Q_\xi)$ is uniquely determined we have $x <_R y$ if and only if $x<_{Q_\xi} y$ whenever $(Q_\alpha, \zeta, \xi, Q_\xi)$ is $<_L$-least such that 
\begin{enumerate}
\item $\xi < \omega_1$, 
\item $Q_\alpha$ is a wellordering of $\mathbb{N}$ isomorphic to $\alpha$,
\item $\zeta$ is an ordinal isomorphic to $R^{\textsc{lth}(Q_\alpha)}_\gamma$,
\item $Q_\xi$ is a wellordering of $\mathbb{N}$ isomorphic to $\xi$,
\item $L_{\omega_2}\models \phi(\zeta, \xi)$.
\end{enumerate}
Now, $R$ is also $\Pi_n^{L_{\omega_2}}(\{\alpha\})$ and so $\Delta_n^{L_{\omega_2}}(\{\alpha\})$. Thus, the length of $R$ is smaller than $\delta$. But $R$ is the length of a witness to the fact that $L_{\omega_2} \models \exists x\, \phi(\gamma,x)$, so this proves the claim that $L_{\delta}\prec_{\Sigma_n} L_{\omega_2}$.

We have seen that either $\sigma_{\Pi_n,\omega_2}$ is of the form $\delta_{n,\omega_2}(\alpha)$ and thus $L_{\sigma_{\Pi_n,\omega_2}} \prec_{\Sigma_n} L_{\omega_2}$, or else $\sigma_{\Pi_n,\omega_2}$ is a limit of such ordinals, in which case we also have $L_{\sigma_{\Pi_n,\omega_2}} \prec_{\Sigma_n} L_{\omega_2}$. 

We now claim that $\sigma_{\Pi_n,\omega_2}$ satisfies the reflection property from the definition of $\sigma_{\Sigma_n,\omega_2}$ in the following stronger form  $(*)$: for all $\gamma_1< \hdots<\gamma_k < \sigma_{\Pi_n,\omega_2}$, whenever $\sigma_{\Pi_n,\omega_2}\in M \in  L_{\omega_2}$, $M \prec_{\Sigma_{n-1}} L_{\omega_2}$, and $M\models\exists x\, \phi(\sigma_{\Pi_n,\omega_2},\gamma_1,\hdots,\gamma_k, x)$, with $\phi \in \Pi_{n-1}$, there exists  $\bar\sigma < \sigma_{\Pi_n,\omega_2}$ such that $L_{\sigma_{\Pi_n,\omega_2}} \models \exists x\, \phi(\bar\sigma,\gamma_1,\hdots,\gamma_k, x)$. Clearly,  $(*)$ and the fact that $L_{\sigma_{\Pi_n,\omega_2}} \prec_{\Sigma_{n-1}} L_{\omega_2}$ together imply the defining property of $\sigma_{\Sigma_n,\omega_2}$, so it will follow that $\sigma_{\Sigma_n,\omega_2} \leq \sigma_{\Pi_n,\omega_2}$. However, both $(*)$ and the fact that $L_{\sigma_{\Pi_n,\omega_2}} \prec_{\Sigma_{n-1}} L_{\omega_2}$ are expressible as $\Pi_n$ formulas over $L_{\omega_2}$, so the strict inequality follows from the definition of $\sigma_{\Pi_n,\omega_2}$.

To prove that $\sigma_{\Pi_n,\omega_2}$ satisfies $(*)$, assume for simplicity that $k = 1$ and let $\gamma = \gamma_1$, $M$, and $\phi$ be as above, so $M\models\exists x\, \phi(\sigma_{\Pi_n,\omega_2}, \gamma, x)$. Then, we have 
\[L_{\omega_2} \models \exists \sigma \,\exists M\, \Big(\sigma, \gamma \in M\wedge M \prec_{\Sigma_{n-1}} V \wedge \text{``$M\models\exists x\, \phi(\sigma,\gamma, x)$''}\Big),\]
where the outermost existential quantifier is witnessed by $\sigma_{\Pi_n,\omega_2}$. This is a $\Sigma_n^{L_{\omega_2}}(\{\gamma\})$ formula, so since $L_{\sigma_{\Pi_n,\omega_2}} \prec_{\Sigma_n} L_{\omega_2}$ it follows that 
\[L_{\sigma_{\Pi_n,\omega_2}} \models \exists \bar\sigma \,\exists M\, \Big(\bar\sigma, \gamma \in M\wedge M \prec_{\Sigma_{n-1}} V \wedge \text{``$M\models\exists x\, \phi(\bar\sigma,\gamma, x)$''}\Big).\]
Letting $\bar\sigma$ and $M$ be as above, we have $M \prec_{\Sigma_{n-1}} L_{\sigma_{\Pi_n,\omega_2}}$ and thus 
\[L_{\sigma_{\Pi_n,\omega_2}}\models \exists x\, \phi(\bar\sigma,\gamma, x),\]
as desired. This proves the lemma. 
\end{proof}

We now conclude this section with:
\begin{proof}[Proof of Theorem \ref{collapse}]
The inequality $\sigma^1_m <\pi^1_m$ holds in $L$ by Cutland's theorem (cf. \S \ref{SectIntro}). The inequality $(\pi^1_m)^L < (\sigma^1_n)^{L[G]}$ holds by Lemma \ref{LemmaSigmaLG}, and $(\sigma^1_n)^{L[G]} < (\pi^1_n)^{L[G]}$ by Lemma \ref{LemmaSigmaPiInLG}. According to Lemma \ref{LemmaReflectionCountable}, we have  $(\pi^1_n)^{L[G]} < \omega_1^L$, which is clearly smaller than $\omega_2^L = (\omega_1)^{L[G]}$.
\end{proof}

\section{Concluding remarks}
We have carried out the first investigations into the structure of reflecting ordinals in non-canonical models of set theory. However, several questions remain open. We mention two. The first one is a strengthening of Corollary \ref{CorollaryOmega2L}:
\begin{question}
Suppose $\omega_1^L < \omega_1$. Must we have $(\sigma^1_3)^L < \sigma^1_3$ and $(\pi^1_3)^L < \pi^1_3$?
\end{question}

The second one concerns the converse, in  a kind of reversal to the results of \S \ref{SectCollapse}. 
\begin{question}
Suppose $(\sigma^1_3)^L < \sigma^1_3$. Must we have $\omega_1^L < \omega_1$?
\end{question}

We also conjecture that the conclusions of Theorem \ref{sacks} and Theorem \ref{lavermiller} remain true for iterations of the corresponding forcing notions:

\begin{conjecture}
The values of $\sigma^1_n$ and $\pi^1_n$ remain unchanged after forcing over $L$ with countable-support iterations of Sacks forcing. The same holds for Miller and Laver forcing.
\end{conjecture}

\subsection{Acknowledgements} 
This work was partially supported by FWF grants I4513 and ESP-3.

\bibliographystyle{abbrv}
\bibliography{References.bib} 
\end{document}